\documentclass[12pt]{amsart}
\usepackage{amssymb}
\usepackage{hyperref}
\usepackage[all,cmtip]{xy}
\usepackage{adjustbox}
\usepackage{enumitem}

\usepackage{multirow}

\makeatletter
\@namedef{subjclassname@2010}{\textup{2010} Mathematics Subject Classification}
\makeatother

\newcommand{\R}{\mathbb{R}}

\newtheorem{theorem}{Theorem}[section]
\newtheorem{lemma}[theorem]{Lemma}
\newtheorem{proposition}[theorem]{Proposition}
\newtheorem{corollary}[theorem]{Corollary}
\theoremstyle{definition}

\newtheorem{example}[theorem]{Example}

\newtheorem{remark}[theorem]{Remark}

\usepackage{microtype}

\usepackage{xcolor}
\usepackage{csquotes}

\begin{document}

\title{General Bourgin--Yang theorems}
\author{Zbigniew B{\l}aszczyk}
\author{Wac{\l}aw Marzantowicz}
\author{Mahender Singh}

\address{Faculty of Mathematics and Computer Science, Adam Mickiewicz University of Pozna{\'n}, ul. Umultowska 87, 61-614 Pozna{\'n}, Poland.}
\email{blaszczyk@amu.edu.pl}
\email{marzan@amu.edu.pl}

\address{Indian Institute of Science Education and Research (IISER) Mohali, Sector 81, Knowledge City, SAS Nagar (Mohali), Post Office Manauli, Punjab 140306, India.}
\email{mahender@iisermohali.ac.in}


\begin{abstract}
We describe a unified approach to estimating the dimension of $f^{-1}(A)$ for any $G$-equivariant map $f \colon X \to Y$ and any closed $G$-invariant subset $A\subseteq Y$ in terms of connectivity of $X$ and dimension of~$Y$, where $G$ is either a cyclic group of order~$p^k$, a $p$-torus ($p$ a prime), or a torus.
\end{abstract}

\maketitle

\section{Introduction}
The celebrated Borsuk--Ulam theorem \cite{BU_original} states that the existence of a continuous map $f \colon S(\mathbb{R}^n) \to S(\mathbb{R}^m)$ between spheres in Euclidean spaces with the property \mbox{$f(-x)=-f(x)$} for all $x \in S(\R^n)$ implies that $n\leq m$. Consequently, if \mbox{$g \colon S(\mathbb{R}^n) \to \mathbb{R}^m$} is a continuous map with that property and $n>m$, then there exists $x_0\in S(\mathbb{R}^n)$ such that $g(x_0)= 0$. Bourgin~\cite{Bourgin} and Yang \cite{Yang1}, \cite{Yang2} showed independently that in this situation the set
\[ Z_g=\{x\in S(\mathbb{R}^n) \,|\, g(x)= 0\}\]
is of dimension at least $n-m-1$.

The Borsuk--Ulam theorem proved to be one of the most useful tools of elementary algebraic topology. For this reason, it has numerous and far reaching extensions and generalizations, and continues to attract attention. Instead of repeating the story of why this is so, we refer the reader to a book by Matou\v{s}ek~\cite{Matousek} and/or a survey by \v{Z}ivaljevi\'{c}~\cite{Zivaljevic}. Let us only mention, very tersely, that Borsuk--Ulam theorems often allow to infer that combinatorial problems have a solution. From this perspective, Bourgin--Yang theorems yield information about the size of the set of those solutions. Interestingly enough, there are not nearly as many papers emphasising the latter point of view. Furthermore, to the best of our knowledge, those which exist rely on having a sphere as the (co)domain of the equivariant map in question (e.g.~Marzantowicz, de~Mattos and dos Santos~\cite{MMS1},~\cite{MMS2}) or are specifically aimed at estimating the size of the so called coincidence sets (e.g.~Munkholm \cite{Munkholm2},~\cite{Munkholm3}, Volovikov~\cite{Volovikov-1},~\cite{Volovikov1}).

In this paper we investigate further extensions of the Bourgin--Yang theorem. These extensions are two-fold in nature. Firstly, we consider symmetries provided by more complicated groups than $\mathbb{Z}_2$, namely we take $G$ to be the cyclic group $\mathbb{Z}_{p^k}$ of order~$p^k$, a $p$-torus~$(\mathbb{Z}_p)^k$ ($p$ a prime in both cases), or a torus $(\mathbb{S}^1)^k$. (See Remark \ref{rem:explanation_G} for a brief explanation on why these are essentially the only reasonable classes of groups to consider.) Secondly, we work with more general spaces than spheres and Euclidean spaces. We merely assume that there exists a $G$-equivariant map $f \colon X \to Y$ with $X$ compact and~$Y$ of finite covering dimension. Then, roughly speaking, our results state that the dimension of $f^{-1}(A)$ for any closed $G$-invariant subspace $A \subseteq Y$ can be measured in terms of the difference of connectivity of $X$ and dimension of $Y$. This idea is based on the work of Clapp and Puppe \cite{Clapp-Puppe}, where a general version of the Borsuk--Ulam theorem is proved.

We were also motivated by a recent exposition of various guises of the topological Tverberg theorem by Blagojevi{\'c} and Ziegler \cite{BZ}. Its prime power version states that for any continuous map $\Delta_N \to \mathbb{R}^m$ with $N \geq (m+1)(p^k-1)$ there exist $p^k$ points in pairwise disjoint faces of the standard $N$-simplex $\Delta_N$ that are mapped to the same point of $\mathbb{R}^m$. We estimate the dimension of the set of such points in terms of the difference between~$N$ and $(m+1)(p^k-1)$, hence obtaining a ``Bourgin--Yang version'' of the topological Tverberg theorem.

Our paper is similar in spirit to that of Volovikov \cite{Volovikov3}. Namely, he defines a numerical $G$-index $i'$ and then derives an ``abstract'' Bourgin--Yang theorem,
\[ i'\big(f^{-1}(A)\big) \geq i'(X) - i'(Y\setminus A). \]
We proceed similarly, with $i'$ replaced with the $(\mathcal{A}, h^*, I)$-length $\ell$ defined by Bartsch~\cite{Bartsch}. The main difference is that we concentrate our efforts on relating $\ell$ to more familiar invariants (such as cohomological or covering dimension) in order to obtain results which resemble the statement of the classical Bourgin--Yang theorem as closely as possible, while Volovikov opts to apply his index to a variety of other problems. As far as applications go, we restrict attention to estimating the size of coincidence sets of maps $X \to \mathbb{R}^m$.

The presented approach provides a unified framework for various types of groups: all of our results arise as variations on a single scheme. It is perhaps worth pointing out that despite that we actually recover the best known results in some classical cases.

The paper is organized as follows.
\begin{itemize}
\item Section \ref{section:preliminatries} consists of preliminaries. Perhaps most importantly, in Subsection \ref{subsection:length} we review the notion and properties of the $(\mathcal{A}, h^*, I)$-length.
\item In Section \ref{section:length_computation} we explain our strategy and carry out certain computations of the $(\mathcal{A}, h^*, I)$-length for various triples $(\mathcal{A}, h^*, I)$.
\item Sections 4 and 5 contain our versions of the Bourgin--Yang theorem for $G=\mathbb{Z}_{p^k}$ and $(\mathbb{S}^1)^k$ (Theorems \ref{thm:BY_cyclic_pk} and Theorem \ref{B-Y for torus}, respectively) and their consequences. Most notably, in the former case we derive estimates on the size of coincidence sets of maps $X \to \R^m$ and discuss how those compare to the previous results of Munkholm~\cite{Munkholm2}, \cite{Munkholm3} and Volovikov \cite{Volovikov-1}.
\item In Section 6 we deal with $G = (\mathbb{Z}_p)^k$. Here the situation is somewhat more complicated and our result is less general (Proposition \ref{B-Y for spaces with the least orbit type}). However, we are still able to recover an estimate that implies the topological Tverberg theorem for prime powers (Theorem \ref{thm:BYTverberg}).
\end{itemize}

\section{Preliminaries}\label{section:preliminatries}

\subsection{Notation.}

We shall use standard notation of transformation group theory, see~\cite{Bredon}. Unless otherwise stated, $G$ denotes a compact Lie group. Throughout the paper $\dim X$ stands for the covering dimension of a space $X$. We note that if $X$ is normal and $A \subseteq X$ is a closed subspace, then $\dim A \leq \dim X$. We will also write
\[ \textnormal{cdim}\,X=\max\{ n\geq 0 \,|\, H^n(X;\mathbb{F}) \neq 0\}, \]
where $H^n(-;\mathbb{F})$ denotes the Alexander--Spanier cohomology with coefficients $\mathbb{F}=\mathbb{Z}_p$ or $\mathbb{F}=\mathbb{Q}$, depending on whether $G=(\mathbb{Z}_p)^k$ or $G=(\mathbb{S}^1)^k$. Since we are working with Alexander--Spanier theory, $\textnormal{cdim}\,X \leq \dim X$.

\subsection{Existence of equivariant maps}

We will make use of the following result from equivariant obstruction theory in order to infer the existence of equivariant maps between certain $G$-spaces.

\begin{theorem}[{\cite[Chapter II, Proposition 3.15]{Dieck}}]\label{obstruction_easy}
Let $n \geq 1$ be an integer. Suppose that $(X,A)$ is a relative $G$-complex with free action on $X \setminus A$ and $Y$ is an $n$-simple and $(n-1)$-connected $G$-space. Then any $G$-equivariant map $A \to Y$ can be extended over the $n$-skeleton of $X$.
\end{theorem}

\subsection{The $(\mathcal{A}, h^*,I)$-length}\label{subsection:length}

Let $\mathcal{A}$ be a set of $G$-spaces, $h^*$ a multiplicative equivariant cohomology theory, and $I\subseteq h^*(\textnormal{pt})$ an ideal. Here $h^*$ is \textit{multiplicative} if it has a natural, bilinear, associative and anti-commutative product
\[ h^m(X,A) \times h^n(X,B) \to h^{m+n}(X,A\cup B) \]
with unit in $h^0(X)$, defined for all excisive $G$-pairs $\{A,B\}$. The \textit{$(\mathcal{A}, h^*, I)$-length} of a $G$-space $X$ is defined to be the smallest integer $k\geq 1$ such that there exist \mbox{$A_1$, \ldots, $A_k \in \mathcal{A}$} with the property that for any $\alpha_i \in I \cap\ker\!\big[h^*(\textnormal{pt}) \to h^*(A_i)\big]$, $1 \leq i \leq k$,
\[ p_X^*(\alpha_1) \smile \cdots \smile p_X^*(\alpha_k) =0,\]
where $p_X \colon X \to \textnormal{pt}$.

Throughout the paper we will make use of either Borel cohomology theory $H_G^*$ modelled on Alexander--Spanier theory, i.e., $H_G^*(X)=H^*(EG \times_G X)$, where $EG$ is the universal $G$-space and $EG \times_G X = (EG \times X)/G$ with respect to the diagonal $G$-action, or equivariant $K$-theory $K_G^*$.

\begin{remark}[{\cite[Observation 5.5]{Bartsch}}]\label{remark:AA'}
Assume that $\mathcal{A}'\subseteq \mathcal{A}$. It is then clear that $(\mathcal{A}, h^*, I)$-length $\leq (\mathcal{A}', h^*, I)$-length. Furthermore, if, additionally, for any $A \in \mathcal{A}$ there exist $A'\in\mathcal{A}'$ and a $G$-equivariant map $A \to A'$, then $(\mathcal{A}, h^*, I)$-length $= (\mathcal{A}', h^*, I)$-length.
\end{remark}

Fix a triple $(\mathcal{A}, h^*, I)$ and write $\ell$ for the $(\mathcal{A}, h^*, I)$-length. In what follows we implicitly assume that all considered spaces are such that $\ell$ is defined for them.

\begin{remark}\label{remark:warning_KG}
Note that even though $K_G^*$ is defined for locally compact $G$-spaces, a $K_G^*$-based length makes sense only for compact $G$-spaces: $K_G^*$ is functorial only for proper maps, hence it is not possible to consider the map induced by $p_X \colon X \to \textnormal{pt}$ on~$K_G^*$ for a non-compact $X$. We mention this specifically because it happens to be a prevalent mistake in \cite{MMS1}. In fact, our work can be seen as an erratum for that paper: even though the arguments therein are flawed, the results are still true.
\end{remark}

\begin{theorem}[{\cite[Theorem 4.7]{Bartsch}}]\label{ell_properties}
The length has the following properties.
\begin{enumerate}
\item[\textnormal{(1)}] If there exists an $h^*$-functorial $G$-equivariant map $X \to Y$, then $\ell(X)\leq \ell(Y)$.
\item[\textnormal{(2)}] Let $A$, $B\subseteq X$ be $G$-invariant subspaces such that
\[ h^*(X,A) \times h^*(X,B) \stackrel{\smile}{\longrightarrow} h^*(X, A\cup B) \]
is defined. If $A \cup B = X$, then $\ell(X)\leq \ell(A)+\ell(B)$.
\item[\textnormal{(3)}] If $h^*=H_G^*$, the ideal $I$ is noetherian and $X$ is paracompact, then any closed $G$-invariant subspace $A \subseteq X$ has an open $G$-invariant neighborhood $U\subseteq X$ such that $\ell(U)=\ell(A)$.
\end{enumerate}
\end{theorem}

\subsection{Cohomology of certain groups}\label{section:cohomology_rings}

We will carry out certain computations in cohomology rings of tori and $p$-tori. Let us recall the relevant information, mainly in order to fix notation.
\begin{itemize}
\item If $G=(\mathbb{Z}_2)^k$, then $H^*(G, \mathbb{Z}_2) \cong \mathbb{Z}_2[w_1, \ldots, w_k]$, where $\deg w_i =1$.
\item If $G=(\mathbb{Z}_p)^k$, $p>2$, then $H^*(G, \mathbb{Z}_p) \cong \Lambda[w_1,\ldots,w_k] \otimes \mathbb{Z}_p[c_1, \ldots, c_k]$, where $\deg w_i=1$ and $\deg c_i=2$.
\item If $G=(\mathbb{S}^1)^k$, then $H^*(G, \mathbb{Q}) \cong \mathbb{Q}[c_1, \ldots, c_k]$, where $\deg c_i=2$.
\end{itemize}

\subsection{Choice of the triple $(\mathcal{A}, h^*, I)$}\label{choice of ideal}
Depending on the group $G$, we set:
\begin{itemize}
\item if $G=(\mathbb{Z}_2)^k$: $h^*= H^*_G(-;\mathbb{Z}_2)$, $I= H_G^*(\textnormal{pt}; \mathbb{Z}_2)$,
\item if $G=(\mathbb{Z}_p)^k$, $p>2$: $h^*= H^*_G(-;\mathbb{Z}_p)$, $I=(c_1, \ldots, c_k)$,
\item if $G=(\mathbb{S}^1)^k$: $h^*= H^*_G(-;\mathbb{Q})$, $I=H^*_G(\textnormal{pt};\mathbb{Q})$.
\end{itemize}
In the three cases above, we will consider the length $\ell$ with respect to the family $\mathcal{A}=\{G/H \,|\, H \subsetneq G \textnormal{ is a closed subgroup}\}$. We will also be interested in a fourth case:
\begin{itemize}
\item if $G=\mathbb{Z}_{p^k}$: $h^*= K^*_G(-)$, $I=K_G(\textnormal{pt})$,
\end{itemize}
but this time with respect to $\mathcal{A}_{r,s} =\{G/H \,|\, \textnormal{$H \subseteq G$ and $r \leq |H| \leq s$}\}$ for two powers $1 \leq r \leq s \leq p^{k-1}$ of $p$. We will write $\ell_s$ for the $(\mathcal{A}_{r,s}, K^*_G, I)$-length. This notation makes sense since $\ell_s$ is independent of the choice of $r$ in view of Remark \ref{remark:AA'}.

\begin{remark}\label{rem:explanation_G}
Our choices for $G$ are not accidental. It is known that if a group $G$ is not an extension of a finite $p$-group of exponent $p$ by a torus, then $G$ does not have the strong Borsuk--Ulam property, i.e., there exist orthogonal $G$-representations $V$ and $W$ with $V^G=W^G=\{0\}$ and a $G$-equivariant map $S(V) \to S(W)$ such that $\dim V > \dim W$. Hence in order to hope for a ``classically flavoured'' general Bourgin--Yang theorem, one has to restrict attention to such extensions. Out of those, however, only tori and $p$-tori are known to have the strong Borsuk--Ulam property.  While $G = \mathbb{Z}_{p^k}$, $k > 1$, does not fall into either of these families, it is known that if $V$ and $W$ are as above, then the existence of a $G$-equivariant map $S(V) \to S(W)$  implies that $\dim V/p^{k-1} \leq \dim W$, so there is a threshold dimension to work with. Consult \cite{Marzantowicz} for a thorough discussion of these and related results for the case of ($p$-)tori, with a caveat that the proof of Lemma~1.2 therein is incomplete, and \cite[Corollary 5.9]{Bartsch} for the case $G =\mathbb{Z}_{p^k}$.
\end{remark}

\subsection{Coincidence sets.}\label{subsection:coincidence}

Let $X$ be a $G$-space and $Y$ any space. Given a map $f\colon X\to Y$, a natural problem is to understand the \textit{coincidence set}
\[ A_f=\{ x\in X \,|\, f(x)=f(gx) \textnormal{ for all $g\in G$}\},\]
e.g. in terms of cohomological dimension or a ``$G$-index'', as in a series of papers by Munkholm \cite{Munkholm2}, \cite{Munkholm3}, and later by Volovikov \cite{Volovikov-1}, \cite{Volovikov1}, \cite{Volovikov3}. In combinatorial applications, $Y$ is usually taken to be a Euclidean space $\mathbb{R}^m$. If $G=\{e,g_1, \ldots, g_r\}$ is a finite group (with a fixed ordering of elements), the following standard argument allows to reduce this problem to a Bourgin--Yang type situation (cf. \cite{MMS1}, \cite{Volovikov1}, \cite{Volovikov2}).

Let $W=\bigoplus_{i=1}^{r+1} \mathbb{R}^m$. Equip $W$ with a $G$-action via the formula
\[ g(w_e, w_{g_1}\ldots, w_{g_r})= (w_g, w_{gg_1}, \ldots, w_{gg_r}). \]
Clearly, this turns $W$ into an orthogonal $G$-representation. The map $\tilde{f}\colon X \to W$ given by
\[ \tilde{f}(x)= \big(f(x), f(g_1x), \ldots, f(g_rx)\big) \]
is then $G$-equivariant. The orthogonal complement ${\perp}W^G$ of $W^G$ is an orthogonal free $G$-representation of dimension $mr$. Now write $\pi^G \colon W \to {\perp}W^G$ for the canonical projection. Since $\Delta W=W^G$ and $\tilde{f}^{-1}(\Delta W) = A_{f}$, the $G$-map $\bar{f} = \pi^G\circ \tilde{f} \colon X \to {\perp}W^G$ has the property that $Z_{\bar{f}}=A_f$.

\section{The $(\mathcal{A}, h^*, I)$-length of certain spaces}\label{section:length_computation}

\subsection{An abstract Bourgin--Yang theorem}\label{subsection:abstract_BY}

Fix a triple $(\mathcal{A}, h^*, I)$ with $h^*=K_G^*$ or $h^* = H_G^*$, and additionally assume that $I$ is noetherian in the latter case. Write $\ell$ for the corresponding length.

\begin{theorem}\label{thm:general_B-Y}
Let $Y$ be a $G$-space and $A\subseteq Y$ a closed $G$-invariant subspace.
\begin{enumerate}
\item[\textnormal{$(K_G^*)$}] If $X$ is a compact $G$-space and {$f \colon X \to Y$} is a $G$-equivariant map, then
\begin{equation*}
\ell\big(f^{-1}(A)\big) \geq \ell(X) - \ell(B)
\end{equation*}
for any compact $G$-invariant subspace $B\subseteq Y$ such that $f(X)\setminus\textnormal{Int}A \subseteq B$.
\item[\textnormal{$(H_G^*)$}] If $X$ is a paracompact $G$-space and {$f \colon X \to Y$} is a $G$-equivariant map, then
\begin{equation*}
\ell\big(f^{-1}(A)\big) \geq \ell(X) - \ell(B)
\end{equation*}
for any $G$-invariant subspace $B\subseteq Y$ such that $f(X)\setminus A \subseteq B$.
\end{enumerate}
\end{theorem}

We note that this sort of result has been formulated by Volovikov \cite{Volovikov3} for a differently defined ``index theory''.

\begin{proof}
$(K_G^*)$ Clearly, $X = f^{-1}(A) \cup \big(X \setminus f^{-1}(\textnormal{Int}A)\big)$. In view of Theorem \ref{ell_properties}, we have:
\begin{itemize}
\item $\ell(X) \leq \ell\big(f^{-1}(A)\big) + \ell\big(X\setminus f^{-1}(\textnormal{Int}A)\big)$,
\item $\ell\big(X\setminus  f^{-1}(\textnormal{Int}A)\big) \leq \ell(B)$, because $f \colon X \setminus f^{-1}(\textnormal{Int} A) \to B$.
\end{itemize}
Combining these two inequalities yields $\ell\big(f^{-1}(A)\big) \geq \ell(X) - \ell(B)$.

$(H_G^*)$ Find an open $G$-invariant neighbourhood $U$ of
$f^{-1}(A)$ such that $\ell(U)=\ell\big(f^{-1}(A)\big)$ by means of Theorem \ref{ell_properties}. Let $V = X - f^{-1}(A)$, so that $X = U \cup V$. Now proceed exactly as above.
\end{proof}

Our immediate goal is to make the above inequalities more accessible for various classes of groups. Note that if $h^*=H_G^*$ and $V$, $W$ are two orthogonal $G$-representations, then applying Theorem \ref{thm:general_B-Y} for $X=S(V)$, $Y=W$, $A=\{0\}$ and $B=W\setminus\{0\}$ yields
\[\ell\big(f^{-1}(0)\big) \geq \ell\big(S(V)\big) - \ell\big(S(W)\big)\]
for any $G$-equivariant map $f \colon S(V) \to W$. In particular, as soon as we manage to relate $\ell$ to more familiar invariants (such as covering or cohomological dimension), we will be able to recover Bourgin--Yang type theorems in the classical framework. A similar thing happens for $h^*=K_G^*$, but this case requires slightly more work and we postpone it for Theorem \ref{thm:BY_spheres_z_pk}.

For the most part we work in a more general setting, with the core assumption being that the base space $X$ is $n$-simple and $(n-1)$-connected for some $n \geq 1$. Our strategy consists of finding:
\begin{enumerate}
\item[(1)] a lower bound depending on $n$ for $\ell$ of such a space $X$, and
\item[(2)] an upper bound for $\ell$ of an arbitrary space in terms of dimension.
\end{enumerate}

\noindent Given a $G$-space $X$, write $\mathcal{A}_X = \{G/G_x \,|\, x \in X\}$. Define $\theta$, $\theta_{-1} \colon \mathbb{N} \to \mathbb{N}$ by setting:
\[ \theta(n) = \begin{cases}
(n+1)/2, & \textnormal{$n$ is odd,}\\
n/2+1, & \textnormal{$n$ is even,}
\end{cases}
\;
\theta_{-1}(n)
= \begin{cases}
(n+1)/2, & \textnormal{$n$ is odd,}\\
n/2, & \textnormal{$n$ is even.}
\end{cases}
\]
This notation will be useful when dealing with various estimates for $\ell$. From now on $n\geq 1$ is a fixed integer and, unless otherwise stated, we are working with the lengths described in Subsection \ref{choice of ideal}.

\subsection{$\mathbf{G=\mathbb{Z}_{p^k}}$} Marzantowicz, de Mattos and dos Santos provided the following upper bound for $\ell_s$:

\begin{theorem}[{\cite[Theorem 3.5]{MMS1}}]\label{thm:ell_vs_dim_KG}
If $X$ is a compact $G$-space and $\mathcal{A}_X \subseteq \mathcal{A}_{r,s}$, then
\[ \ell_s(X)\leq \theta(\dim X). \]
\end{theorem}

\noindent On the other hand, we have:

\begin{proposition}\label{prop:highly_connected_KG}
If $X$ is a compact, $n$-simple and $(n-1)$-connected $G$-space, then
\[ \ell_s(X) \geq \left\lceil\frac{\big(\theta_{-1}(n)-1\big)r}{s}\right\rceil + 1. \]
\end{proposition}

\begin{proof}
Apply Theorem \ref{obstruction_easy} to produce a $G$-equivariant map $S(V) \to X$, where $V$ is an orthogonal free $G$-representation of complex dimension $\theta_{-1}(n)$. Monotonicity of $\ell_s$ then shows that $\ell_s(X) \geq \ell_s\big(S(V)\big)$, and the latter is estimated in \cite[Theorem 5.8]{Bartsch}.
\end{proof}

\subsection{$\mathbf{G=(\mathbb{Z}_p)^k}$ or $\mathbf{G=(\mathbb{S}^1)^k}$}

We reiterate that $\mathbb{F}=\mathbb{Z}_p$ or $\mathbb{F}=\mathbb{Q}$, depending on whether $G=(\mathbb{Z}_p)^k$ or $G=(\mathbb{S}^1)^k$. The latter case is included in the statements of the next two results by taking $p=\infty$ to mean $\mathbb{S}^1$.

Write $BG$ for the classifying space of $G$. We will make use of the following observation without further ado: if $H \subseteq G$ is a closed subgroup, then the map $H_G^*(\textnormal{pt}) \to H_G^*(G/H)$ coincides with the map $H^*(BG;\mathbb{F}) \to H^*(BH;\mathbb{F})$ induced by the inclusion $H \to G$.

\begin{proposition}\label{prop:ell_vs_dim_zp}
Let $k=1$ and $X$ be a $G$-space which is either compact or paracompact of finite covering dimension. If $X$ is fixed-point free, then
\[
\ell(X) \leq \begin{cases}
\textnormal{cdim}\, X +1, & p=2, \\
\theta(\textnormal{cdim}\,X), & \textnormal{$p>2$ finite,}\\
\theta_{-1}(\textnormal{cdim}\,X), & p=\infty.
\end{cases}
\]
\end{proposition}

\begin{proof}
Given a closed subgroup $H \subsetneq G$, we have 
\[I\cap \ker\!\big[H_G^*(\textnormal{pt}) \to H_G^*(G/H)\big] = I \cap \tilde{H}^*(BG;\mathbb{F}).\] 
The conclusion follows immediately from the structure of cohomology of~$G$, or more precisely the choice of ideal $I$ (see Subsection \ref{choice of ideal}), and the fact that $\textnormal{cdim}\, EG \times_G X \leq \textnormal{cdim}\,X$ (or $\leq \textnormal{cdim}\,X-1$ if $p=\infty$), as shown in \mbox{\cite[Theorem 1.4]{Bredon2}}.
\end{proof}

\begin{remark}
We cannot drop the assumption that $k=1$ in the statement of Proposition~ \ref{prop:ell_vs_dim_zp}: see Example \ref{ex:p-torus_bad_behaviour}, and also Subsection \ref{subsect:strong_normalization} for a more elucidative explanation. It is, however, perhaps worth pointing out that Proposition \ref{prop:ell_vs_dim_zp} does hold in a slightly more general situation, namely for $X$ a paracompact $G$-space of finite cohomological dimension in the sense of sheaf cohomology (\textit{ibid.}).
\end{remark}

\begin{proposition}\label{prop:estimate from below2}
If $X$ is a $G$-space such that $H^i(X; \mathbb{F})=0$ for $0<i<n$, then
\[ \ell(X) \geq \begin{cases}
n+1, & p=2,\\
\theta(n), & p>2.
\end{cases} \]
\end{proposition}

\begin{proof}
Assume that $G=(\mathbb{Z}_p)^k$ first. Choose a subgroup $H
\subsetneq G$ and note that
\[ \tilde{H}^*\big(B(G/H); \mathbb{Z}_p\big) \subseteq \ker\!\big[H^*_G(\textnormal{pt}) \to H^*_G(G/H)\big] .\]
This is because the extension $0 \to H \to G \to G/H \to 0$ is split. In particular, $I \cap  \ker\!\big[H^*_G(\textnormal{pt}) \to H^*_G(G/H)\big]$ contains a non-zero element, say $\alpha_H$, in the first (second, respectively) gradation for $p=2$ ($p>2$).

Since the length of a path component of a space never exceeds that of the whole space, we can assume without loss of generality that $X$ is path-connected. Consider the Serre spectral sequence with $\mathbb{Z}_p$-coefficients $\{E^{*,*}\}$ of the Borel fibration
\[X \hookrightarrow EG\times_G X \stackrel{p_X}{\longrightarrow} BG.\]
It is well-known that $p_X^* \colon H^i(BG;\mathbb{Z}_p) \to H^i(EG\times_G X;\mathbb{Z}_p)$ can be expressed as the edge homomorphism
\[ H^i(BG;\mathbb{Z}_p) \cong E_2^{i,0} \twoheadrightarrow E_\infty^{i,0} \hookrightarrow H^i(EG\times_G X;\mathbb{Z}_p), \]
where $E_\infty^{i,0}$ is a quotient of $E_2^{i,0}$ by images of successive differentials and $E_2^{i,0} \twoheadrightarrow E_\infty^{i,0}$ is the quotient map. But the assumption on $X$ assures that $E_\infty^{i,0} = E_2^{i,0}$ in the range $0 \leq i \leq n$ and, consequently, $p_X^*$ is an isomorphism for $i<n$ and a monomorphism for $i=n$. This means that for any sequence $H_1$, \ldots, $H_r \subsetneq G$ we have
\[ p_X^*(\alpha_{H_1}) \smile \cdots \smile p_X^*(\alpha_{H_r})\neq 0,\]
where $r=n$ or $r=\theta(n)-1$, depending on whether $p=2$ or $p>2$.

If $G=(\mathbb{S}^1)^k$ and $H \subsetneq G$ is a closed subgroup, the extension \mbox{$0\to H \to G \to G/H \to 0$} need not be split, as $H \cong (\mathbb{S}^1)^l \times \Gamma$, where $0 \leq l < k$ and $\Gamma$ is a finite abelian group. However, in this case we are working with rational coefficients, thus we can just as well assume that $H$ is a torus and proceed exactly as above.
\end{proof}

\section{A Bourgin--Yang theorem for $G = \mathbb{Z}_{p^k}$}

We can immediately state:

\begin{theorem}\label{thm:BY_cyclic_pk}
Let $X$ and $Y$ be $G$-spaces, with $X$ compact, $n$-simple and $(n-1)$-connected, and $Y$ paracompact. If $\mathcal{A}_X$, $\mathcal{A}_Y \subseteq \mathcal{A}_{r,s}$, then for any $G$-equivariant map $f \colon X \to Y$ one has
\[ \theta\big(\!\dim f^{-1}(A)\big) \geq \left\lceil\frac{\big(\theta_{-1}(n)-1\big)r}{s}\right\rceil - \theta\big(\!\dim (Y\setminus \textnormal{Int}A)\big) +1, \]
where $A \subseteq Y$ is a closed $G$-invariant subspace.
\end{theorem}

\begin{proof}
Combining Theorem \ref{thm:general_B-Y} with Theorem \ref{thm:ell_vs_dim_KG} and Proposition \ref{prop:highly_connected_KG} yields
\[ \theta\big(\!\dim f^{-1}(A)\big) \geq \left\lceil\frac{\big(\theta_{-1}(n)-1\big)r}{s}\right\rceil - \theta\big(\!\dim(f(X)\setminus \textnormal{Int}A)\big) +1. \]
Now use the fact that $\dim\!\big(f(X) \setminus \textnormal{Int}A\big) \leq \dim(Y \setminus \textnormal{Int}A)$ in order to conclude the proof.
\end{proof}

Let $V$ and $W$ be complex orthogonal $G$-representations such that $V^G = W^G = \{0\}$. Set $X=S(V)$, $Y=W$ and $A=\{0\}$. Note that we cannot apply Theorem \ref{thm:BY_cyclic_pk} verbatim in order to obtain the estimate
\begin{equation}\label{MMS1_estimate}
\dim Z_f \geq 2\left(\left\lceil\frac{(\dim_{\mathbb{C}}V- 1)r}{s}\right\rceil - \dim_{\mathbb{C}}W\!\right)\tag{$\clubsuit$}
\end{equation}
of \cite[Theorem 3.6]{MMS1}. We show, however, that it is actually possible to recover that result, even though its original proof is flawed (see Remark \ref{remark:warning_KG}).

\begin{theorem}\label{thm:BY_spheres_z_pk}
Let $V$ and $W$ be as above. If $\mathcal{A}_{V\setminus\{0\}}$, $\mathcal{A}_{W\setminus\{0\}} \subseteq \mathcal{A}_{r,s}$, then for any $G$-equivariant map $f \colon S(V) \to W$ one has
\[\theta(\dim Z_f) \geq \left\lceil\frac{(\dim_{\mathbb{C}}V-1)r}{s}\right\rceil - \dim_{\mathbb{C}}W +1. \]
Consequently, $($\ref{MMS1_estimate}$)$ holds.
\end{theorem}

\begin{proof}
Consider $W$ as a metric space with the usual Euclidean metric. Given $\varepsilon>0$, write $D_{\varepsilon}$ for the closed $\varepsilon$-disk centered at $0 \in W$. Since $f\big(S(V)\big)$ is compact, there exists $R>0$ such that $f\big(S(V)\big)\subseteq D_R$, hence $f \colon S(V) \setminus f^{-1}(\textnormal{Int} D_{\varepsilon}) \to D_R \setminus \textnormal{Int} D_{\varepsilon}$ for any $0<\varepsilon<R$. But $K_G^*(D_R \setminus \textnormal{Int} D_{\varepsilon}) \cong K_G^*\big(S(W)\big)$ by the obvious equivariant deformation argument.  Similarly as before, Theorem \ref{thm:general_B-Y} together with Theorem \ref{thm:ell_vs_dim_KG} and Proposition \ref{prop:highly_connected_KG} give
\[ \ell_s\big(f^{-1}(D_{\varepsilon})\big) \geq \left\lceil\frac{(\dim_{\mathbb{C}}V-1)r}{s}\right\rceil - \dim_{\mathbb{C}}W +1, \textnormal{ $0 < \varepsilon < R$.} \]
Thus in order to conclude the proof it suffices to show that $\ell_s\big(f^{-1}(D_{\varepsilon_0})\big) = \ell_s\big(f^{-1}(0)\big)$ for some $0 < \varepsilon_0 < R$ and use Theorem \ref{thm:ell_vs_dim_KG} once more.

Of course $\ell_s\big(f^{-1}(0)\big) \leq \ell_s\big(f^{-1}(D_{\varepsilon})\big)$ for any $0 < \varepsilon < R$ by monotonicity of $\ell_s$. Let $\ell_s\big(f^{-1}(0)\big)=n$ and take $A_1$, \ldots, $A_n \in \mathcal{A}_{r,s}$ as in the definition of length. Since $I=K_G(\textnormal{pt})\cong \mathbb{Z}[x]/(1-x^{p^k})$ is noetherian, we can choose finitely many generators
\[ \alpha_{ij} \in I \cap\ker\!\big[K_G^*(\textnormal{pt}) \to K_G^*(A_i)\big], \textnormal{ $1\leq i\leq n$, $1 \leq j \leq m_i$.} \]
Then, given any tuple $J=(j_1,\ldots, j_n)$ with $1 \leq j_i \leq m_i$ and $\alpha_J = \alpha_{1j_1}\smile \cdots \smile \alpha_{nj_n}$, we have $p_{f^{-1}(0)}^*(\alpha_J)=0$. Since $K_G^*\big(f^{-1}(0)\big)\cong \varinjlim K_G^*\big(f^{-1}(D_{\varepsilon})\big)$ by continuity of $K_G^*$, there exists $0 <\varepsilon_J < R$ such that $p_{f^{-1}(D_{\varepsilon_J})}^*(\alpha_J)=0$. It is easy to see that $\varepsilon_0 = \min\{\varepsilon_J \,|\, J\}$ has $\ell_s\big(f^{-1}(D_{\varepsilon_0})\big)\leq n$.
\end{proof}

\begin{remark}
The conclusion of Theorem \ref{thm:BY_spheres_z_pk} remains true \textit{mutatis mutandis} with~$S(V)$ replaced with any compact $G$-space $X$ which is $n$-simple, $(n-1)$-connected and has $\mathcal{A}_X \subseteq \mathcal{A}_{r,s}$.
\end{remark}

\begin{corollary}\label{cor:volovikov_munkholm_zpk}
Let $p>2$. If $X$ is a fixed-point free $G$-space which is compact, $n$-simple and $(n-1)$-connected, then for any map $f \colon X \to \mathbb{R}^m$ one has
\[ \theta(\dim A_f) \geq \left\lceil\frac{\theta_{-1}(n)-1}{p^{k-1}}\right\rceil - \frac{(p^k-1)m}{2} +1. \]
\end{corollary}

\begin{proof}
Proceed as explained in Subsection \ref{subsection:coincidence} in order to obtain a $G$-equivariant map $\bar{f} \colon X \to {\perp}W^G$. Note that $\mathcal{A}_{{\perp}W^G\setminus\{0\}} = \mathcal{A}_{1,p^{k-1}}$, so that ${\perp}W^G$ admits a complex structure and $\dim_{\mathbb{C}}{\perp}W^G = \frac{1}{2}(p^k-1)$. Now argue as in the proof of Theorem \ref{thm:BY_spheres_z_pk} with $r=1$, $s=p^{k-1}$.
\end{proof}

Let $X=S^{2n-1}$, $n \geq 1$, equipped with a fixed-point free $G$-action. Our result shows in particular that if
\[ 2n-1\geq mp^{k-1}(p^k-1)-2p^{k-1}+3,\]
then any map $f \colon S^{2n-1} \to \mathbb{R}^m$ has $A_f \neq \emptyset$. In general, this is a considerably worse estimate than the one provided by the ``$\textnormal{mod}\,p^a$ Borsuk--Ulam theorem'' of Munkholm~\cite{Munkholm3}. (The exception being the case $k=1$, where the two estimates coincide. And even then this is not the best known result; see the discussion immediately after Theorem \ref{thm:BY_Zp}.) Note, however, that there is room for improvement in our approach: we believe that plugging in the actual value of $\ell_{p^{k-1}}\big(S({\perp}W^G)\big)$ into the inequality in Corollary \ref{cor:volovikov_munkholm_zpk}, as opposed to replacing it by $\dim_{\mathbb{C}}{\perp}W^G = \frac{1}{2}(p^k-1)m$, allows to recover Munkholm's estimate. This in fact happens in the few cases that we have been able to calculate $\ell_{p^{k-1}}\big(S({\perp}W^G)\big)$ by using \textsc{Macaulay2}, which boils down to finding the least integer $l\geq 1$ such that $\big(1-x^{p^{k-1}}\big)^l$ is contained in the ideal $\big(1-x^{p^k}, e({\perp}W^G)\big)$ of $\mathbb{Z}[x]$, where $e({\perp}W^G)$ stands for the Euler class of the complex $G$-vector bundle ${\perp}W^G \to \textnormal{pt}$. For this reason we conjecture that
\[ \ell_{p^{k-1}}\big(S({\perp}W^G)\big) = \frac{1}{2}km(p-1).\]

We can derive the following ``Clapp--Puppe version''  of the Borsuk--Ulam theorem for $G=\mathbb{Z}_{p^k}$, which originally was proved for tori and $p$-tori (see \cite[Theorem~6.4]{Clapp-Puppe}). 

\begin{theorem}\label{Borsuk-Ulam for cyclic}
Let $X$ and $Y$ be compact $G$-spaces, with $X$ $n$-simple and $(n-1)$-connected. If $\mathcal{A}_Y \subseteq \mathcal{A}_{r,s}$ and
\[\left\lceil\frac{\big(\theta_{-1}(n)-1\big)r}{s}\right\rceil +1 > \theta(\dim Y),\]
then there is no $G$-equivariant map $X \to Y$.
\end{theorem}

\begin{proof}
If a $G$-equivariant map $X\to Y$ exists, then $\ell_s(X) \leq \ell_s(Y)$ by monotonicity of~$\ell_s$ and $\big\lceil\big(\theta_{-1}(n)-1\big)r/s\big\rceil+1\leq\theta(\dim Y)$ by Theorem \ref{thm:ell_vs_dim_KG} and Proposition \ref{prop:highly_connected_KG}.
\end{proof}

Let us state the corresponding results for $k=1$ separately. In fact, we can readily obtain slightly better results if we replace $K_{\mathbb{Z}_p}^*$ with $H_{\mathbb{Z}_p}^*$. Up until the end of this section $\ell$ stands for the $\big(\{\mathbb{Z}_p\}, H_{\mathbb{Z}_p}^*(-;\mathbb{Z}_p), I\big)$-length, where $I=H_{\mathbb{Z}_2}^*(\textnormal{pt};\mathbb{Z}_2)$ if $p=2$ and $I=(c_1)$ if $p>2$ (see Subsection \ref{section:cohomology_rings}).

\begin{theorem}\label{thm:BY_Zp}
Let $X$ and $Y$ be $\mathbb{Z}_p$-spaces with $X$ compact such that $H^i(X;\mathbb{Z}_p)=0$ for $0<i<n$, and $Y$ paracompact of finite covering dimension. If $X$ is fixed-point free, then for any $\mathbb{Z}_p$-equivariant map $f \colon X \to Y$ one has:
\begin{enumerate}[leftmargin=1.5cm]
\item[\textnormal{(}$p=2$\textnormal{)}] $\textnormal{cdim}\,f^{-1}(A) \geq n - \textnormal{cdim}(Y\setminus \textnormal{Int}A) -1$,
\item[\textnormal{(}$p>2$\textnormal{)}] $\theta\big(\textnormal{cdim} f^{-1}(A)\big) \geq \theta(n) - \theta\big(\textnormal{cdim}(Y \setminus \textnormal{Int}A)\big)$,
\end{enumerate}
where $A\subseteq Y$ is a closed $G$-invariant subspace such that $Y^{\mathbb{Z}_p}\subseteq \textnormal{Int}A$.
\end{theorem}

\begin{proof}
Combine Theorem \ref{thm:general_B-Y} with Propositions \ref{prop:ell_vs_dim_zp} and \ref{prop:estimate from below2}.
\end{proof}

\begin{remark}\label{remark:interior}
The reason $\textrm{Int}A$ appears in the statement of the above theorem is that the set $Y\setminus A$ is not necessarily paracompact of finite covering dimension. If that is the case, e.g. when $Y$ is a separable metric space, then $\textrm{Int}A$ can be replaced with $A$. The same is true for Theorem \ref{B-Y for torus} and Proposition \ref{B-Y for spaces with the least orbit type}.
\end{remark}

If $X$ is a $\mathbb{Z}_p$-space as in the statement of Theorem \ref{thm:BY_Zp}, then the size of the coincidence set~$A_f$ of a map $f \colon X \to \mathbb{R}^m$ can be estimated as follows:
\[ \textnormal{cdim}\,A_f \geq \begin{cases}
n-m, & p=2,\\
n-(p-1)m-1, & p>2.
\end{cases}
\]
To see this, use the reduction described in Subsection \ref{subsection:coincidence} before applying Theorem~\ref{thm:BY_Zp}, tweaked as explained in Remark \ref{remark:interior}, for $A=\{0\}$. For $p>2$ this is slightly weaker (we~are off by ``$-1$'') than Munkholm's ``$\textnormal{mod}\,p$ Bourgin--Yang theorem'' \cite{Munkholm2} and Vo\-lo\-vi\-kov's \cite[Theorem 1]{Volovikov-1}. However, in the special case when $X$ is a $\textnormal{mod}\,p$ cohomology sphere, our approach can be improved as follows in order to match their estimates.

Let $\ell_1$ be the length corresponding to the triple $\big(\{\mathbb{Z}_p\}, H_{\mathbb{Z}_p}^*(-;\mathbb{Z}_p), H_{\mathbb{Z}_p}^*(\textnormal{pt};\mathbb{Z}_p)\big)$. Then $\ell_1$ can be described as follows (see \cite[Example 4.5]{Bartsch}):
\[
\ell_1(X) = \begin{cases}
\ell(X), & \textnormal{$p_X^*(c^{\ell(X)-1}w)=0$,}\\
\ell(X)+1, & \textnormal{otherwise}.
\end{cases}
\]
Now define $\ell'(X)=\ell(X) + \ell_1(X) - 1$. This is not length in the sense of definition in Subsection \ref{subsection:length}, but it still inherits some of its properties. In particular:

\begin{lemma}
Let $X$ be a fixed-point free $\mathbb{Z}_p$-space with $p>2$.
\begin{enumerate}
\item[\textnormal{(1)}] If $\ell_1(X)=\ell(X)+1$, then analogues of Theorems \ref{ell_properties} and \ref{thm:general_B-Y} hold for $\ell'$.
\item[\textnormal{(2)}] If $X$ is compact or paracompact of finite covering dimension, then $\ell'(X)\leq \textnormal{cdim}\,X + 1$.
\item[\textnormal{(3)}] If $X$ is a compact $\textnormal{mod}\,p$ cohomology $n$-sphere, then $\ell'(X)=n+1$.
\end{enumerate}
\end{lemma}

\begin{proof}
(1) It is clear that properties (1) and (3) of Theorem \ref{ell_properties} hold even without any extra assumptions, and it is explained in \cite[Remark 4.14]{Bartsch} that (2) holds provided that $\ell_1(X)=\ell(X)+1$. We can now repeat the proof of Theorem \ref{thm:general_B-Y} word for word.

(2) This is a straightforward consequence of Proposition \ref{prop:ell_vs_dim_zp}.

(3) Note that $p>2$ forces $n$ to be odd. By means of Propositions \ref{prop:ell_vs_dim_zp} and \ref{prop:estimate from below2}, $\ell(X)=(n+1)/2$. It is also clear that $\ell_1(X)\leq (n+1)/2+1$. Furthermore, in view of Theorem \ref{obstruction_easy}, there exists a $\mathbb{Z}_p$-equivariant map $S(V) \to X$, where $V$ is an orthogonal free $\mathbb{Z}_p$-representation of complex dimension $(n+1)/2$, so that $\ell_1(X)\geq \ell_1\big(S(V)\big)$. The latter is equal to $(n+1)/2+1$ by \cite[Remark 5.4]{Bartsch}. Consequently, $\ell_1(X)=(n+1)/2+1$ and $\ell'(X) = \ell(X) + \ell_1(X)-1 = n+1$.
\end{proof}

In effect, we can prove:

\begin{proposition}\label{prop:coincidence_zp}
If $X$ is a fixed-point free $\mathbb{Z}_p$-space which is a compact $\textnormal{mod}\,p$ cohomology $n$-sphere, then for any map $f \colon X \to \mathbb{R}^m$ one has
\[ \textnormal{cdim}\,A_f \geq n-(p-1)m. \]
\end{proposition}

We note that similar methods can be used to estimate sizes of ``$(H,G)$-coincidence sets'' (see \cite{GJP} for a definition), as exhibited in \cite{deMattos-Souza}.

\section{A Bourgin--Yang theorem for $(\mathbb{S}^1)^k$}\label{section:B-Y_torus}

When $G=(\mathbb{S}^1)^k$, we can still obtain fairly general results via a straightforward reduction to the case $k=1$:

\begin{lemma}[{\cite[Lemma 4.2.1]{AlldayPuppe}}]\label{Z_p acting free}
Let $X$ be a $G$-space with finitely many orbit types and $X^G=\emptyset$. Then there exists a subgroup $H\subseteq G$, $H\cong \mathbb{S}^1$, such that $X^H=\emptyset$.
\end{lemma}

\noindent Thus an additional assumption we will work with is that the target $G$-space $Y$ has only finitely many orbit types. This is well-known to be satisfied e.g. if~$Y$ is an invariant compact subspace of a $G$-manifold, a $G$-CW-complex, or a $G$-ENR.

\begin{theorem}\label{B-Y for torus}
Let $X$ and $Y$ be $G$-spaces, with $X$ compact such that $H^i(X;\mathbb{Q})=0$ for $0<i<n$, and $Y$ paracompact of finite covering dimension and with finitely many orbit types. If $Y$ is fixed-point free, then for any $G$-equivariant map $f \colon X \to Y$ one has
\[ \theta_{-1}\big(\textnormal{cdim} f^{-1}(A)\big) \geq \theta(n) - \theta_{-1}\big(\textnormal{cdim} (Y \setminus  {\rm Int} A)\big), \]
where $A \subseteq Y$ is a closed $G$-invariant subspace.
\end{theorem}

\begin{proof}
If $k=1$, use Theorem \ref{thm:general_B-Y} with Propositions \ref{prop:ell_vs_dim_zp} and \ref{prop:estimate from below2}. If $k>1$, consider $f \colon X \to Y$ as an $H$-equivariant map between $H$-spaces for $H\cong \mathbb{S}^1$ given by Lemma~\ref{Z_p acting free}, so that $Y$, and hence also $X$, are fixed-point free $H$-spaces.
\end{proof}

The resulting version for representation spheres is as follows. (In particular, we recover the estimate of \cite[Theorem 2.1]{MMS2} in a more streamlined manner.)

\begin{corollary}
If $V$ and $W$ are complex orthogonal $G$-representations such that $V^G=W^G = \{0\}$, then for any $G$-equivariant map $f \colon S(V) \to W$ one has
\[ \theta_{-1}\big(\textnormal{cdim} f^{-1}(0)\big) \geq \dim_{\mathbb{C}}V - \dim_{\mathbb{C}} W. \]
Consequently,

\[ \textnormal{cdim} f^{-1}(0) \geq 2(\dim_{\mathbb{C}}V - \dim_{\mathbb{C}} W)-1. \]
\end{corollary}

One can also derive the corresponding version of the Borsuk--Ulam theorem by using Lemma \ref{Z_p acting free} and proceeding analogously as in the proof of Theorem \ref{Borsuk-Ulam for cyclic}.

\begin{theorem}\label{thm:BU-torus}
Let $X$ and $Y$ be fixed-point free $G$-spaces, with $X$ such that \mbox{$H^i(X;\mathbb{Q})=0$} for $0<i<n$ and $Y$ paracompact of finite covering dimension and with finitely many orbit types. If $\theta(n) > \theta_{-1} (\textnormal{cdim}\,Y)$, then there is no $G$-equivariant map $X \to Y$.
\end{theorem}

Using the same notation as above, \cite[Theorem 6.4]{Clapp-Puppe} states that if $n > \textrm{cdim}\,Y$, then there is no $G$-equivariant map $X \to Y$. Theorem \ref{thm:BU-torus} sharpens this result in the following sense. Depending on the values of $n$ and $\textrm{cdim}\,Y$, unveiling $\theta$'s yields:\medskip

\begin{center}
\begin{tabular}{|cc|cc|}
\hline
& & \multicolumn{2}{c|}{$n$}\cr
& & \textnormal{odd} & \textnormal{even} \cr\hline
\multirow{2}{*}{$\textrm{cdim}\,Y$} & odd & $n > \textrm{cdim}\,Y\phantom{{}-1}$ & $n>\textrm{cdim}\,Y -1$  \cr
& even & $n>\textrm{cdim}\,Y -1$ & $n>\textrm{cdim}\,Y -2$ \cr\hline
\end{tabular}
\end{center}\medskip

\noindent Therefore, if we e.g. know \textit{a priori} that $n$ is even, we can (slightly) relax the assumption to $n>\textrm{cdim}\,Y -1$ and still infer the non-existence of a $G$-equivariant map $X \to Y$.

\section{A Bourgin--Yang theorem for $(\mathbb{Z}_p)^k$}

\subsection{General results} This case is significantly more difficult than the previous ones. It does not succumb to a reduction similar to the one used in Section \ref{section:B-Y_torus}, and we also cannot hope for a satisfactory upper bound on $\ell(X)$ in terms of dimension of $X$.

\begin{example}\label{ex:p-torus_bad_behaviour}
Let $H \subseteq G$ be a sub-$p$-torus of rank $k-1$. It is clear that for any $K \subsetneq G$, $K\neq H$, there exists $\alpha \in I \cap \ker\!\big[H_G^*(\textnormal{pt}) \to H_G^*(G/K)\big]$ such that $(p_{G/H})^*(\alpha)\neq 0$. Since $(p_{X\sqcup Y})^* = (p_X^*, p_Y^*)$, this shows that for any family $H_1$, \ldots, $H_n \subseteq G$ of distinct sub-$p$-tori we have $\ell\big(\!\bigsqcup_{i=1}^n G/H_i\big)= n$, but $\dim\!\big(\!\bigsqcup_{i=1}^n G/H_i\big) =0$.
\end{example}

This sort of behaviour is triggered by the fact that the length $\ell$ we work with in this case does not satisfy the so called ``strong normalization property''. We briefly discuss this in Subsection \ref{subsect:strong_normalization}. In general, there is the following upper bound:
\[ \ell(X) \leq (\dim X + 1)\max_{H\subsetneq G} c(H), \]
where $c(H)$ is the number of connected components of $X^H\!/G$, see \cite[Proposition 4.2]{Bartsch}. However, this is far from being useful from our point of view. Therefore we state a Bourgin--Yang type theorem under an additional assumption.

\begin{proposition}\label{B-Y for spaces with the least orbit type}
Let $X$ and $Y$ be $G$-spaces, with $X$ compact and \mbox{$H^i(X;\mathbb{Z}_p)=0$} for $0<i<n$, and $Y$ paracompact of finite covering dimension. If there exists a subgroup $H\subseteq G$, $H\cong \mathbb{Z}_p$, such that $Y$ is a fixed-point free $H$-space, then for any $G$-equivariant map $f \colon X \to Y$ one has:
\begin{enumerate}[leftmargin=1.5cm]
\item[\textnormal{(}$p=2$\textnormal{)}] $\textnormal{cdim}\,f^{-1}(A) \geq n - \textnormal{cdim}(Y \setminus\textnormal{Int}A) -1$,
\item[\textnormal{(}$p>2$\textnormal{)}] $\theta\big(\textnormal{cdim} f^{-1}(A)\big) \geq \theta(n) - \theta\big(\textnormal{cdim}(Y \setminus \textnormal{Int}A)\big)$,
\end{enumerate}
where $A \subseteq Y$ is a closed $G$-invariant subspace.
\end{proposition}

\begin{proof}
Note that both spaces are fixed-point free with respect to the $H$-action and it suffices to apply Theorem \ref{thm:BY_Zp}.
\end{proof}

\begin{remark}
The paper by Marzantowicz, de Mattos and dos Santos \cite{MMS2}, published when this work was already complete, contains a version of the Bourgin--Yang theorem for maps between spheres in orthogonal $G$-representations without an assumption of the existence of the subgroup $H$. The methods used therein are different.
\end{remark}

We can also estimate the size of a coincidence set in the following setting:

\begin{proposition}\label{coincidence_zpk}
If $X$ is a free $G$-space which is compact and such that $H^i(X;\mathbb{Z}_p)=0$ for $0<i<n$, then for any map $f \colon X \to \mathbb{R}^m$ one has:
\begin{enumerate}[leftmargin=1.5cm]
\item[\textnormal{(}$p=2$\textnormal{)}] $\dim A_f \geq n-(2^k-1)m$,
\item[\textnormal{(}$p>2$\textnormal{)}] $\theta(\dim A_f) \geq \theta(n)- (p^k-1)m/2$.
\end{enumerate}
\end{proposition}

\begin{proof}
($p=2$) The argument in Subsection \ref{subsection:coincidence} yields a $G$-map \mbox{$\bar{f} \colon X \to {\perp}W^G$}, where ${\perp}W^G$ is a fixed-point free $G$-representation of dimension $(2^k-1)m$. By Theorem \ref{thm:general_B-Y} and Proposition \ref{prop:estimate from below2},
\[ \ell(A_f) \geq n+1 - \ell\big(S({\perp}W^G)\big). \]
But $A_f$ is a free $G$-space, hence $EG \times_G A_f \simeq A_f/G$, so that \mbox{$\ell(A_f) \leq \textnormal{cdim}\,A_f/G+1$}, and $\textnormal{cdim}\,A_f/G \leq \dim A_f/G \leq \dim A_f$, where the last inequality is a consequence of \mbox{\cite[Proposition 3.7]{Deo-Tripathi}}. On the other hand, $\ell\big(S({\perp}W^G)\big) = \dim {\perp}W^G = (2^k-1)m$ in view of \cite[Theorem~5.2]{Bartsch}.

($p>2$) The argument is similar, with two minor modifications: plug in $\theta$ wherever necessary, and in this case $\ell\big(S({\perp}W^G)\big) = \frac{1}{2}\dim {\perp}W^G = \frac{1}{2}(p^k-1)m$ (\textit{ibid.}).
\end{proof}

\subsection{A ``Bourgin--Yang version'' of the topological Tverberg theorem for prime powers}

The reader is referred to a recent survey by Blagojevi\'c and Ziegler \cite{BZ} for more information on the topological Tverberg theorem.

In what follows, $\Delta_N$ denotes the standard $N$-dimensional simplex.

\begin{theorem}\label{thm:BYTverberg}
Let $k$, $m \geq 1$ be integers, $p$ a prime, and $N\geq (m+1)(p^k-1)$. If $f\colon \Delta_N \to \R^m$ is a continuous map, then there exist $p^k$ pairwise disjoint faces \mbox{$\sigma_1$, \ldots, $\sigma_{p^k}$} of $\Delta_N$ such that the set $C=\big\{(x_1, \ldots, x_{p^k}) \in \sigma_1 \times \cdots \times \sigma_{p^k} \,\big|\, f(x_1) = \cdots = f(x_{p^k})\big\}$ satisfies:
\begin{enumerate}[leftmargin=1.5cm]
\item[\textnormal{(}$p=2$\textnormal{)}] $\dim C \geq N-(m+1)(2^k-1)$,
\item[\textnormal{(}$p>2$\textnormal{)}] $\theta(\dim C) \geq \theta(N-p^k+1)-(p^k-1)m/2$.
\end{enumerate}
In particular, in both cases, $f(\sigma_1) \cap \cdots \cap f(\sigma_{p^k}) \neq \emptyset$.
\end{theorem}

\begin{proof}
Consider the $p^k$-fold $2$-wise \textit{deleted product} $\mathbf{\Delta}=(\Delta_N)^{p^k}_{\Delta(2)}$ of $\Delta_N$, i.e.,
\[ \mathbf{\Delta}=\big\{(x_1, \ldots, x_{p^k}) \in \sigma_1 \times \cdots \times \sigma_{p^k} \subseteq (\Delta_N)^{p^k} \,|\, \sigma_i \cap \sigma_j = \emptyset \textnormal{ for all $i \neq j$}\big\}. \]
It is well-known that $\mathbf{\Delta}$ is an $(N-p^k+1)$-dimensional $(N-p^k)$-connected CW complex (see e.g. \cite[Theorem 3.4]{BZ}). Equip it with a $G$-action analogously as in Subsection \ref{subsection:coincidence} and consider a map $f_1 \colon \mathbf{\Delta} \to \R^m$ given by $f_1(x_1,\ldots, x_{p^k}) = f(x_1)$. If $p=2$, Proposition \ref{coincidence_zpk} immediately implies that $\dim A_{f_1} \geq N-(m+1)(p^k-1)$. Since $\sigma_1 \times \cdots \times \sigma_{p^k}$ is closed in $(\Delta_N)^{p^k}$ for any collection $\sigma_1$, \ldots, $\sigma_{p^k}$ of faces of $\Delta_N$ and \[ A_{f_1} = \bigcup A_{f_1}\cap \sigma_1 \times \cdots \times \sigma_{p^k}, \]
where the union is taken over all $\sigma_1 \times \cdots \times \sigma_{p^k} \subseteq \mathbf{\Delta}$, it follows that
\[ \dim A_{f_1} \cap \sigma_1 \times \cdots \times \sigma_{p^k} \geq N-(m+1)(p^k-1) \]
for some $\sigma_1 \times \cdots \times \sigma_{p^k} \subseteq \mathbf{\Delta}$. These are the searched for faces of $\Delta_N$, and the set $C$ is exactly $A_{f_1} \cap \sigma_1 \times \cdots \times \sigma_{p^k}$. In particular, $C$ is non-empty, hence $f(\sigma_1) \cap \cdots \cap f(\sigma_{p^k}) \neq \emptyset$.

If $p>2$, the same argument shows that
\[ \theta(\dim A_{f_1} \cap \sigma_1 \times \cdots \times \sigma_{p^k}) \geq \theta(N-p^k+1) - (p^k-1)m/2 \]
for some  $\sigma_1 \times \cdots \times \sigma_{p^k} \subseteq \mathbf{\Delta}$. Unveiling $\theta$ reveals that $\dim C \geq 0$, so that $C$ is non-empty in this case as well. Indeed, if $d = N - (m+1)(p^k-1)$, then
\[
\dim C \geq
\begin{cases}
d+1, & \textnormal{$\dim C$ is odd and $N-p^k+1$ is even,}\\
d, & \textnormal{$\dim C$ and $N-p^k+1$ are of the same parity,}\\
d-1, & \textnormal{$\dim C$ is even and $N-p^k+1$ is odd.}
\end{cases}
\]
In the last case, however, $N$ is necessarily an odd integer; in particular, $d\geq 1$.
\end{proof}

\subsection{Remark on the strong normalization property}\label{subsect:strong_normalization}

In this subsection $G$ stands for a compact Lie group, $\mathcal{A} \subseteq \{ G/H \,|\, H \subsetneq G \textnormal{ is a closed subgroup}\}$ and $h^*$ denotes a continuous and multiplicative equivariant cohomology theory. \textit{Continuity} of $h^*$ means that for any closed $G$-invariant subspace $A$ of a compact $G$-space $X$ we have $h^*(A)=\varinjlim h^*(\mathcal{U})$, where the direct limit is taken over all open $G$-neighbourhoods $\mathcal{U}$ of $A$ in $X$. Furthermore, $\ell$ stands for the $(\mathcal{A}, h^*, I)$-length for some choice of an ideal $I \subseteq h^*(\textnormal{pt})$.

The length $\ell$ is said to have the \textit{strong normalization property} provided that
\[ \ell(A_1 \sqcup \cdots \sqcup A_k)=1 \textnormal{ for any $k\geq 1$ and any $A_1$, \ldots, $A_k \in \mathcal{A}$.} \]
Clearly, the length we discussed for $G=(\mathbb{Z}_p)^k$ lacks strong normalization. The following result sheds some light on why this is troublesome from the point of view of our approach. (We omit its proof as it is quite technical and not really relevant to the main theme of this paper.)

\begin{proposition}
If $\ell$ has the strong normalization property, then \mbox{$\ell(X) \leq \dim X+1$} for any compact $G$-space $X$ of finite orbit type such that $\mathcal{A}_X \subseteq \mathcal{A}$. If $G$ is a finite group, the converse also holds.
\end{proposition}

\noindent\textbf{Acknowledgements.}
The first and second authors have been supported by the National Science Centre under grants 2014/12/S/ST1/00368 and 2015/19/B/ST1/01458, respectively. The third author has been supported by DST INSPIRE Scheme IFA-11MA-01/2011 and MTR/2017/000018.

We also acknowledge that in his PhD thesis N. A. Silva \cite{Nelson} obtained results which partially coincide with ours: he was mainly interested in the case when~$X$ is a $\textnormal{mod}\,p$ cohomology sphere.

\end{document}